\documentclass[12pt]{amsart}
\usepackage[utf8]{inputenc}
\usepackage{hyperref}
\usepackage[final]{showkeys} 
\usepackage{amssymb}
\usepackage{amsmath}
 \usepackage{amsthm}
\usepackage{enumitem}
\usepackage{tikz -inet}
\usetikzlibrary{positioning,shapes.misc, patterns,arrows,decorations,decorations.pathreplacing, snakes}

\newtheorem{theorem}{Theorem}[section]
\newtheorem{hypothesis}[theorem]{Hypothesis}

\newtheorem{definition}[theorem]{Definition}

\newtheorem{lemma}[theorem]{Lemma}
\newtheorem{fact}[theorem]{Fact}
\newtheorem{remark}[theorem]{Remark}
\newtheorem{corollary}[theorem]{Corollary}

\newcommand{\fct}[2]{{}^{#1}#2}

\newcommand{\K}{\mathcal{K}}
\newcommand{\Ksatp}[1]{\K^{#1\text{-sat}}}
\newcommand{\rest}{\upharpoonright}
\newcommand{\Union}{\bigcup}
\DeclareMathOperator{\tp}{ga-tp}
\DeclareMathOperator{\id}{id}
\DeclareMathOperator{\Aut}{Aut}
\DeclareMathOperator{\cf}{cf}
\newcommand{\seq}[1]{\langle #1 \rangle}
\newcommand{\gaS}{\operatorname{ga-S}}
\newcommand{\gS}{\gaS}
\newcommand{\gtp}{\tp}
\DeclareMathOperator{\LS}{LS}

\newcommand{\C}{\mathfrak{C}}

\newcommand{\sea}{\mathfrak{C}}
\newcommand{\hanf}[1]{h (#1)}

\newcommand{\ba}{\bar{a}}
\newcommand{\bb}{\bar{b}}
\newcommand{\bc}{\bar{c}}

\newbox\noforkbox \newdimen\forklinewidth
\forklinewidth=0.3pt \setbox0\hbox{$\textstyle\smile$}
\setbox1\hbox to \wd0{\hfil\vrule width \forklinewidth depth-2pt
 height 10pt \hfil}
\wd1=0 cm \setbox\noforkbox\hbox{\lower 2pt\box1\lower
2pt\box0\relax}
\def\unionstick{\mathop{\copy\noforkbox}\limits}

\setbox0\hbox{$\textstyle\smile$}
\setbox1\hbox to \wd0{\hfil{\sl /\/}\hfil} \setbox2\hbox to
\wd0{\hfil\vrule height 10pt depth -2pt width
               \forklinewidth\hfil}
\newbox\doesforkbox
\setbox\doesforkbox\hbox{\lower 0pt\box1 \lower
2pt\box2\lower2pt\box0\relax}

\newcommand{\nf}{\unionstick}

\newcommand{\Sbs}{\gS^\text{bs}}
\newcommand{\s}{\mathfrak{s}}

\def\lta{<}
\def\lea{\le}

\title[Categorical AECs with amalgamation]{On the structure of categorical abstract elementary classes with amalgamation}
\date{\today}

\author{Monica M. VanDieren}

\address{Department of Mathematics\\
Robert Morris University\\
Moon Township PA 15108}

\email{vandieren@rmu.edu}

\author{Sebastien Vasey}

\address{Department of Mathematical Sciences, Carnegie Mellon University, Pittsburgh, Pennsylvania, USA}
\email{sebv@cmu.edu}
\urladdr{http://math.cmu.edu/\textasciitilde svasey/}
\thanks{The second author is supported by the Swiss National Science Foundation.}
\begin{document}

\begin{abstract}

For $\K$ an abstract elementary class with amalgamation and no maximal models, we show that categoricity in a high-enough cardinal implies structural properties such as the uniqueness of limit models and the existence of good frames. This improves several classical results of Shelah.

\begin{theorem}\label{abstract-thm}
  Let $\mu \ge \LS (\K)$. If $\K$ is categorical in a $\lambda \ge \beth_{\left(2^{\mu}\right)^+}$, then:

  \begin{enumerate}
    \item\label{abstract-1} Whenever $M_0, M_1, M_2 \in \K_\mu$ are such that $M_1$ and $M_2$ are limit over $M_0$, we have $M_1 \cong_{M_0} M_2$.
    \item\label{abstract-2} If $\mu > \LS (\K)$, the model of size $\lambda$ is $\mu$-saturated.
    \item\label{abstract-3} If $\mu \ge \beth_{(2^{\LS (\K)})^+}$ and $\lambda \ge \beth_{\left(2^{\mu^+}\right)^+}$, then there exists a type-full good $\mu$-frame with underlying class the saturated models in $\K_\mu$.
  \end{enumerate}
\end{theorem}

Our main tool is the symmetry property of splitting (previously isolated by the first author). The key lemma deduces symmetry from failure of the order property.

\end{abstract}

\maketitle

\tableofcontents

\section{Introduction}

The guiding conjecture for the classification of abstract elementary classes (AECs) is Shelah's categoricity conjecture.  Most progress towards this conjecture has been made under the assumption that the categoricity cardinal is a successor, e.g.\ \cite{sh394, tamenessthree, tamelc-jsl}\footnote{Recently, the second author has proved a categoricity transfer theorem without assuming that the categoricity cardinal is a successor, but assuming that the class is universal \cite{ap-universal-v6} (other partial results not assuming categoricity in a successor cardinal are in \cite{indep-aec-v4} and \cite{shelahaecbook}).}. In this paper, we assume the amalgamation property and no maximal models and deduce new structural results without having to assume that the categoricity cardinal is a successor, or even ``high-enough'' cofinality.

Consider an AEC $\K$ with amalgamation and no maximal models which is categorical in a cardinal $\lambda > \LS (\K)$. Then $\K$ is stable in every cardinal below $\lambda$ \cite[Claim 1.7.(b)]{sh394}; so if $\cf (\lambda) > \LS (\K)$, then the model of size $\lambda$ is $\cf (\lambda)$-saturated\footnote{In the sense of Galois types, i.e.\ $M$ is \emph{$\lambda$-saturated} if every Galois type over a model of size less than $\lambda$ contained in $M$ is realized in $M$. This is called ``$\lambda$-Galois-saturated'' by some authors, but here we always drop the ``Galois'' (and similarly for other concepts such as stability).}. In particular, if $\lambda$ is regular then the model of size $\lambda$ is saturated. Baldwin \cite[Problem D.1.(2)]{baldwinbook09} has asked if this generalizes to any ``sufficiently large'' cardinal $\lambda$, so let us discuss what happens if we have no control over the cofinality of $\lambda$.

One strategy to show saturation of the model of size $\lambda$ is to show that $\K$ is stable in $\lambda$. However an example of Hart and Shelah \cite{hs-example} yields (for an arbitrary $k < \omega$) a sentence $\psi_k \in L_{\omega_1, \omega}$ categorical in $\aleph_0, \aleph_1, \ldots, \aleph_k$, but not stable in $\aleph_k$, \cite {bk-hs} or see \cite[Corollary 26.5.4]{baldwinbook09}. Therefore it is not in general true that categoricity in $\lambda$ implies stability in $\lambda$.

On the other hand, it \emph{is} true if we assume a locality property for Galois types: tameness. This is due to the second author and combines the stability transfer in \cite{ss-tame-toappear-v3} and the Shelah-Villaveces theorem \cite{shvi635}.

\begin{fact}
 If $\K$ is a $\LS (\K)$-tame AEC with amalgamation, has no maximal models, and is categorical in a $\lambda > \LS (\K)$, then $\K$ is stable in every cardinal. In particular, the model of size $\lambda$ is saturated.
\end{fact}
\begin{proof}
  By the Shelah-Villaveces theorem (see Fact \ref{shvi}), $\K$ is $\LS (\K)$-superstable (see Definition \ref{ss assm}). Let $\mu \ge \LS (\K)$. By \cite[Proposition 10.10]{indep-aec-v4}, $\K$ is $\mu$-superstable, so in particular stable in $\mu$.
\end{proof}

In this paper, we do \emph{not} assume tameness: we show that we can instead take $\lambda$ sufficiently big (this is (\ref{abstract-1}) of the theorem in the abstract, see Corollary \ref{cor-mu-sat} for the proof):

\begin{theorem}\label{main-thm-sat}
  Let $\K$ be an AEC with amalgamation and no maximal models and let $\mu > \LS (\K)$. If $\K$ is categorical in a $\lambda \ge \beth_{(2^{\mu})^+}$, then the model of size $\lambda$ is $\mu$-saturated. 
\end{theorem}

Note that we only obtain $\mu$-saturation, not full saturation (the slogan is that categoricity cardinals above $\beth_{(2^{\mu})^+}$ behave as if they had cofinality at least $\mu$). However if $\lambda = \beth_\lambda$ this gives full saturation (this is used to obtain a downward categoricity transfer, see Corollary \ref{easy-downward}). Moreover $\mu$-saturation is enough for many applications, as many of the results of \cite{sh394} only assume categoricity in a $\lambda$ with $\cf (\lambda) > \mu$ (for a fixed $\mu \ge \LS (\K)$). For example, we show how to obtain \emph{weak tameness} (i.e.\ tameness over saturated models) from categoricity in a big-enough cardinal (this is Theorem \ref{weak-tameness-from-categ}). We can then build a local notion of independence: a good $\mu$-frame (this is (\ref{abstract-3}) of the theorem in the abstract, see Corollary \ref{good-frame-categ} for a proof): 

\begin{theorem}
  Let $\mu \ge \beth_{(2^{\LS (\K)})^+}$. Assume that $\K$ is categorical in some $\lambda \ge \beth_{(2^{\mu^+})^+}$. Then there exists a type-full good $\mu$-frame with underlying class the saturated models in $\K_\mu$.
\end{theorem}

This improves on \cite[Theorem 7.4]{ss-tame-toappear-v3}, which assumed categoricity in a successor (and a higher Hanf number bound). This also (partially) answers \cite[Remark 4.9.(1)]{sh394} which asked whether there is a parallel to forking in categorical AECs with amalgamation. 

The key to the proof of Theorem \ref{main-thm-sat} is a close study of the symmetry property for splitting, identified by the first author in \cite{vandieren-symmetry-v1}. There it was shown (assuming superstability in $\mu$) that symmetry of $\mu$-splitting is equivalent to the continuity of reduced towers of size $\mu$, which itself implies uniqueness of limit models in $\mu$. It was also shown that symmetry of $\mu$-splitting follows from categoricity in $\mu^+$. In \cite{vv-symmetry-transfer-v2}, we improved this by only requiring categoricity in a $\lambda$ of cofinality bigger than $\mu$:

\begin{fact}[Corollary 5.2 in \cite{vv-symmetry-transfer-v2}]\label{sym-categ-fact}
  Let $\K$ be an AEC with amalgamation and no maximal models. Let $\mu \ge \LS (\K)$. Assume that $\K$ is categorical in a cardinal $\lambda$ with $\cf (\lambda) > \mu$. Then $\K$ is $\mu$-superstable and has $\mu$-symmetry. In particular \cite[Theorem 0.1]{vv-symmetry-transfer-v2}, it has uniqueness of limit models in $\mu$: for any $M_0, M_1, M_2 \in \K_\mu$, if $M_1$ and $M_2$ are limit over $M_0$, then $M_1 \cong_{M_0} M_2$.
\end{fact}

Here we replace the cofinality assumption on the categoricity cardinal with the assumption that the categoricity cardinal is big enough (this also proves (\ref{abstract-2}) of the theorem in the abstract, see Theorem \ref{sym-categ}):

\begin{theorem}\label{sym-from-categ-0}
  Let $\K$ be an AEC with amalgamation and no maximal models. Let $\mu \ge \LS (\K)$. If $\K$ is categorical in a $\lambda \ge \beth_{(2^{\mu})^+}$, then $\K$ is $\mu$-superstable and has $\mu$-symmetry. In particular, $\K$ has uniqueness of limit models in $\mu$.
\end{theorem}
\begin{remark}\label{sym-from-categ-0-rmk}
  This gives a proof (assuming amalgamation and a high-enough categoricity cardinal) of the (in)famous \cite[Theorem 3.3.7]{shvi635}, where a gap was identified in the first author's Ph.D.\ thesis. The gap was fixed assuming categoricity in $\mu^+$ in \cite{vandierennomax, nomaxerrata} (see also the exposition in \cite{gvv-toappear-v1_2}). In \cite[Corollary 6.10]{bg-v9}, this was improved to categoricity in an arbitrary $\lambda > \mu$ provided that $\mu$ is big-enough and the class satisfies strong locality assumptions (full tameness and shortness and the extension property for coheir). In \cite[Theorem 7.11]{ss-tame-toappear-v3}, only tameness was required but the categoricity had to be in a $\lambda$ with $\cf (\lambda) > \mu$. Still assuming tameness, this is shown for categoricity in any $\lambda \ge \beth_{(2^{\mu})^+}$ in \cite[Theorem 7.1]{bv-sat-v3} and in any $\lambda > \mu$ in \cite[Theorem 6.4]{vv-symmetry-transfer-v2}. Finally without tameness the result was proved for any $\lambda$ with $\cf(\lambda) > \mu$ in \cite[Theorem 0.1]{vv-symmetry-transfer-v2}.
\end{remark}

The proof of Theorem \ref{sym-from-categ-0} is technical but conceptually not hard: we show that a failure of $\mu$-symmetry would give the order property, which in turn would imply instability below the categoricity cardinal. The idea of using the order property to prove symmetry of an independence relation is due to Shelah, \cite[Theorem 6.10(ii)]{sh54} or see \cite[Theorem III.4.13]{shelahfobook}. In \cite{bgkv-v2}, an abstract generalization of Shelah's proof to any independence notion satisfying extension and uniqueness was given. Here we adapt the proof of \cite{bgkv-v2}  to splitting. This uses the extension property of splitting for models of different sizes from \cite{vv-symmetry-transfer-v2}.

In general, we obtain that an AEC with amalgamation categorical in a high-enough cardinal has many structural properties that were previously only known for AECs categorical in a cardinal of high-enough \emph{cofinality}, or even just in a successor. This improves several classical results from Shelah's milestone study of categorical AECs with amalgamation \cite{sh394}.

This paper was written while the second author was working on a Ph.D.\ thesis under the direction of Rami Grossberg at Carnegie Mellon University and he would like to thank Professor Grossberg for his guidance and assistance in his research in general and in this work specifically. 

\section{Background}

Throughout this paper, we assume:

\begin{hypothesis}\label{ap-hyp}
  $\K$ is an AEC with amalgamation.
\end{hypothesis}

For convenience, we fix a big-enough monster model $\sea$ and work inside $\sea$.   This is possible since by Remark \ref{jep remark}, we will have the joint embedding property in addition to the amalgamation property for models of the relevant cardinalities.

Many of the pre-requisite definitions and notations used in this paper can be found in \cite{gvv-toappear-v1_2}.  Here we recall the more specialized concepts that we use explicitly.
We begin by recalling the definition of non-splitting, a notion of independence from \cite[Definition 3.2]{sh394}. 

\begin{definition}\label{def:splitting}
A type $p \in \gS (N)$ does not $\mu$-split over $M$ if and only if for any $N_1, N_2 \in \K_\mu$ such that $M \lea N_\ell \lea N$ for $\ell = 1,2$, and any $f: N_1 \cong_M N_2$, we have $f (p \rest N_1) = p \rest N_2$
\end{definition}

The definition of superstability below is already implicit in \cite{shvi635} and has since then been studied in several papers, e.g.\ \cite{vandierennomax, gvv-toappear-v1_2, indep-aec-v4, bv-sat-v3, gv-superstability-v2, vv-symmetry-transfer-v2}. We will use the formulation from \cite[Definition 10.1]{indep-aec-v4}:

\begin{definition}\label{ss assm}
$\K$ is \emph{$\mu$-superstable} (or \emph{superstable in $\mu$})
if:

  \begin{enumerate}
    \item $\mu \ge \LS (\K)$.
    \item $\K_\mu$ is nonempty, has joint embedding, and no maximal models.
    \item $\K$ is stable in $\mu$\footnote{That is, $|\gS (M)| \le \mu$ for all $M \in \K_\mu$. Some authors call this ``Galois-stable''.}, and:
    \item\label{split assm} $\mu$-splitting in $\K$ satisfies the following
  locality (sometimes called continuity) and ``no long splitting chains''
  properties:
  
%
%
For any limit ordinal $\alpha < \mu^+$, for every sequence $\langle M_i\mid i<\alpha\rangle$ of
  models of cardinality $\mu$ with $M_{i+1}$ universal over $M_i$ and for every $p\in\gaS(\bigcup_{i < \alpha} M_i)$, we have that:
\begin{enumerate}
\item\label{locality} If for every $i<\alpha$, the type $p\restriction
  M_i$ does not $\mu$-split over $M_0$, then $p$ does not $\mu$-split over
  $M_0$.
\item\label{no long splitting chain} There exists $i<\alpha$ such that $p$
does not $\mu$-split over $M_i$.
\end{enumerate}
\end{enumerate}

\end{definition}
\begin{remark}\label{jep remark}
  By our global hypothesis of amalgamation (Hypothesis \ref{ap-hyp}), if $\K$ is $\mu$-superstable, then $\K_{\ge \mu}$ has joint embedding.
\end{remark}
\begin{remark}
  By the weak transitivity property of $\mu$-splitting \cite[Proposition 3.7]{ss-tame-toappear-v3}, condition (\ref{locality}) could have been omitted (i.e.\ it follows from the rest).
\end{remark}

The main tool of this paper is the concept of symmetry over limit models which was identified in \cite{vandieren-symmetry-v1}:
\begin{definition}\label{sym defn}
An abstract elementary class  exhibits \emph{symmetry for $\mu$-splitting} (or \emph{$\mu$-symmetry} for short) if  whenever models $M,M_0,N\in\K_\mu$ and elements $a$ and $b$  satisfy the conditions \ref{limit sym cond}-\ref{last} below, then there exists  $M^b$  a limit model over $M_0$, containing $b$, so that $\tp(a/M^b)$ does not $\mu$-split over $N$.  See Figure \ref{fig:sym}.
\begin{enumerate} 
\item\label{limit sym cond} $M$ is universal over $M_0$ and $M_0$ is a limit model over $N$.
\item\label{a cond}  $a\in M\backslash M_0$.
\item\label{a non-split} $\tp(a/M_0)$ is non-algebraic and does not $\mu$-split over $N$.
\item\label{last} $\tp(b/M)$ is non-algebraic and does not $\mu$-split over $M_0$. 
   
\end{enumerate}
\end{definition}

\begin{figure}[h]
\begin{tikzpicture}[rounded corners=5mm, scale=3,inner sep=.5mm]
\draw (0,1.25) rectangle (.75,.5);
\draw (.25,.75) node {$N$};
\draw (0,0) rectangle (3,1.25);
\draw (0,1.25) rectangle (1,0);
\draw (.85,.25) node {$M_0$};
\draw (3.2, .25) node {$M$};
\draw[color=gray] (0,1.25) rectangle (1.5, -.5);
\node at (1.1,-.25)[circle, fill, draw, label=45:$b$] {};
\node at (2,.75)[circle, fill, draw, label=45:$a$] {};
\draw[color=gray] (1.75,-.25) node {$M^{b}$};
\end{tikzpicture}
\caption{A diagram of the models and elements in the definition of symmetry. We assume the type $\tp(b/M)$ does not $\mu$-split over $M_0$ and $\tp(a/M_0)$ does not $\mu$-split over $N$.  Symmetry implies the existence of $M^b$ a limit model over $M_0$ containing $b$, so that $\tp(a/M^b)$  does not $\mu$-split over $N$.} \label{fig:sym}
\end{figure}
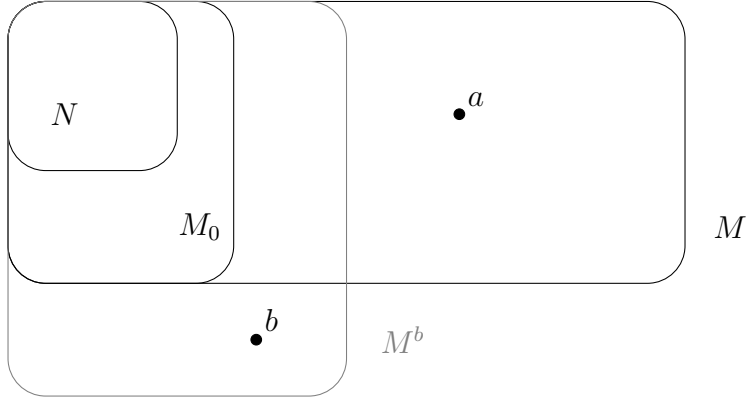

We recall a few results of the first author showing the importance of the symmetry property:

\begin{fact}[Theorem 5 in \cite{vandieren-symmetry-v1}]\label{sym-uq-lim}
  If $\K$ is $\mu$-superstable and has $\mu$-symmetry, then for any $M_0, M_1, M_2 \in \K_\mu$, if $M_1$ and $M_2$ are limit models over $M_0$, then $M_1 \cong_{M_0} M_2$.
\end{fact}

For $\lambda > \LS (\K)$, we will write $\Ksatp{\lambda}$ for the class of $\lambda$-saturated models in $\K_{\ge \lambda}$ (we order it with the strong substructure relation induced by $\K$). By \cite{vandieren-sat-v1} superstability and symmetry together imply that the union of certain chains of saturated models is saturated. This has an easier formulation in \cite[Theorem 6.6]{vv-symmetry-transfer-v2}:

\begin{fact}\label{union-sat-monica}
  Assume $\K$ is $\mu$-superstable, $\mu^+$-superstable, and has $\mu^+$-symmetry. Then $\Ksatp{\mu^+}$ is an AEC with $\LS (\Ksatp{\mu^+}) = \mu^+$.
\end{fact}

The following downward transfer of the symmetry property was the key result of \cite{vv-symmetry-transfer-v2}:

\begin{fact}[Theorem 1.1 in \cite{vv-symmetry-transfer-v2}]\label{sym-transfer}
  Let $\lambda > \mu \ge \LS (\K)$. Assume that $\K$ is superstable in every $\chi \in [\mu, \lambda]$. If $\K$ has $\lambda$-symmetry, then $\K$ has $\mu$-symmetry.
\end{fact}

It will be convenient to use the following independence notion. A minor variation (where ``limit over'' is replaced by ``universal over'') appears in \cite[Definition 3.8]{ss-tame-toappear-v3}. 

\begin{definition}[Definition 3.1 in \cite{vv-symmetry-transfer-v2}]\label{mu-forking-def}
  Let $M_0 \lea M \lea N$ be models with $M_0 \in \K_\mu$. We say a type $p \in \gS (N)$ \emph{explicitly does not $\mu$-fork over $(M_0, M)$} if:

    \begin{enumerate}
      \item $M$ is limit over a model containing $M_0$.
      \item $p$ does not $\mu$-split over $M_0$.
    \end{enumerate}

    We say that \emph{$p$ does not $\mu$-fork over $M$} if there exists $M_0$ so that $p$ explicitly does not $\mu$-fork over $(M_0, M)$.
\end{definition}
\begin{remark}
  Assuming $\mu$-superstability, the relation ``$p$ does not $\mu$-fork over $M$'' is very close to defining an independence notion with the properties of forking in a first-order superstable theory (i.e.\ a good $\mu$-frame, see Section \ref{good-frame-sect}). In fact using tameness (or just, as we will show in Section \ref{good-frame-sect}, weak tameness) it can be used to do precisely that, see \cite{ss-tame-toappear-v3}.
\end{remark}



$\mu$-Forking has the following properties:

\begin{fact}\label{mu-forking-props}
  Let $\mu \ge \LS (\K)$. Assume that $\K$ is stable in $\mu$.
  \begin{enumerate}
    \item\label{mu-forking-1}\cite[I.4.10, I.4.12]{vandierennomax} Extension and uniqueness in $\mu$: Let $M \lea N$ be in $\K_\mu$, and $p \in \gS (M)$ explicitly does not $\mu$-fork over $(M_0, M)$, then there exists a unique $q \in \gS (N)$ extending $p$ so that $q$ explicitly does not $\mu$-fork over $(M_0, M)$. Moreover $q$ is algebraic if and only if $p$ is.
    \item\cite[Proposition 4.3]{vv-symmetry-transfer-v2} Continuity above $\mu$: Assume that $\K$ is $\mu$-superstable. 

  Suppose $\langle M_i\in\K_{\ge \mu}\mid i< \delta\rangle$ is an increasing sequence of models so that, for all $i < \delta$, $M_{i+1}$ is universal over $M_i$. Let $p\in\gaS(\Union_{i<\delta}M_i)$.
  If $p\restriction M_i$ does not $\mu$-split over $M_0$ for each $i<\delta$, then $p$ does not $\mu$-split over $M_0$.
    \item\label{mu-forking-3} Extension above $\mu$: Let $\lambda > \mu$. Let $M \lea N$ be in $K_{[\mu,\lambda]}$. Let $p \in \gS (M)$ be such that $p$ explicitly does not $\mu$-fork over $(M_0, M)$. If $\K$ is superstable in every $\chi \in [\mu, \lambda]$, then there exists $q \in \gS (N)$ extending $p$ and explicitly not $\mu$-forking over $(M_0, M)$. Moreover $q$ is algebraic if and only if $p$ is.
  \end{enumerate}
\end{fact}
\begin{proof}[Proof of (\ref{mu-forking-3})]
  By induction on $\|N\|$. Let $a$ realize $p$. If $\|N\| = \|M\|$ this is given by \cite[Proposition 4.4]{vv-symmetry-transfer-v2}. If $\|M\| < \|N\|$, build $\seq{N_i \in \K_{\|M\| + |i|} : i \le \|N\|}$ increasing continuous such that $N_0 = M$, $N_{i + 1}$ is limit over $N_i$, and $\gtp (a / N_i)$ explicitly does not $\mu$-fork over $(M_0, M)$. This is possible by the induction hypothesis and the continuity property of forking. Now $N_\lambda$ is $\|N\|$-universal over $N_0 = M$, so let $f: N \xrightarrow[M]{} N_\lambda$. Let $q := f^{-1} (\gtp (a / f[N]))$. It is easy to check that $q$ is as desired.
\end{proof}

\section{Symmetry from no order property}

In this section we show (assuming enough instances of superstability) that the negation of symmetry implies the order property, and hence contradicts stability. This is similar to \cite[Theorem 5.14]{bgkv-v2}, but  due to the intricate definition of the symmetry property for splitting, some technical details have to be handled. We first give an equivalent definition of symmetry. Recall that in \cite[Definition 3.3]{vv-symmetry-transfer-v2}, we gave three variations on the symmetry property: 

\begin{enumerate}
  \item The uniform $\mu$-symmetry, which is essentially Definition \ref{sym defn} (and in fact is formally equivalent to it).
  \item The non-uniform $\mu$-symmetry, which weakens the conclusion of uniform $\mu$-symmetry by ``changing'' the model $N$ that $\gtp (a / M^b)$ does not $\mu$-split over.
  \item The weak non-uniform $\mu$-symmetry which strengthens the hypotheses of non-uniform $\mu$-symmetry by requiring that $\gtp (b / M)$ does not $\mu$-fork, instead of $\mu$-split over $M_0$.
\end{enumerate}

There is a fourth possible variation, the weak uniform $\mu$-symmetry property, which strengthens the hypotheses of uniform $\mu$-symmetry similarly to the weak non-uniform $\mu$-symmetry, but leaves the conclusion unchanged.  For clarity we have underlined the differences between the weak and non-weak definitions.

\begin{definition}\label{many-syms}
  Let $\mu \ge \LS (\K)$.
  \begin{enumerate}
    \item $\K$ has \emph{uniform $\mu$-symmetry} if for any limit models $N, M_0, M$ in $\K_\mu$ where $M$ is limit over $M_0$ and $M_0$ is limit over $N$, if \underline{$\gtp (b / M)$ does not $\mu$-split over $M_0$}, $a \in |M|$, and $\gtp (a / M_0)$ explicitly does not $\mu$-fork over $(N, M_0)$, there exists $M_b \in \K_{\mu}$ containing $b$ and limit over $M_0$ so that $\gtp (a / M_b)$ explicitly does not $\mu$-fork over $(N, M_0)$.
    \item $\K$ has \emph{weak uniform $\mu$-symmetry} if for any limit models $N, M_0, M$ in $\K_\mu$ where $M$ is limit over $M_0$ and $M_0$ is limit over $N$, if \underline{$\gtp (b / M)$ does not $\mu$-fork over $M_0$}, $a \in |M|$, and $\gtp (a / M_0)$ explicitly does not $\mu$-fork over $(N, M_0)$, there exists $M_b \in \K_{\mu}$ containing $b$ and limit over $M_0$ so that $\gtp (a / M_b)$ explicitly does not $\mu$-fork over $(N, M_0)$.  See Figure \ref{fig:weak uniform sym}.
  \end{enumerate}
\end{definition}

As noted above, the difference is that
 in the non-weak version, $\gtp (b / M)$ does not $\mu$-split only over $M_0$, but in the weak version
we require that $\gtp (b / M)$ does not $\mu$-\emph{fork} over $M_0$.  Thus in the weak version, there exists $M_0'$ such that $M_0$ is limit over $M_0'$ and $\gtp (b / M)$ does not $\mu$-split over $M_0'$ (but we do not ask that $M_0'$ has any relationship to $N$).
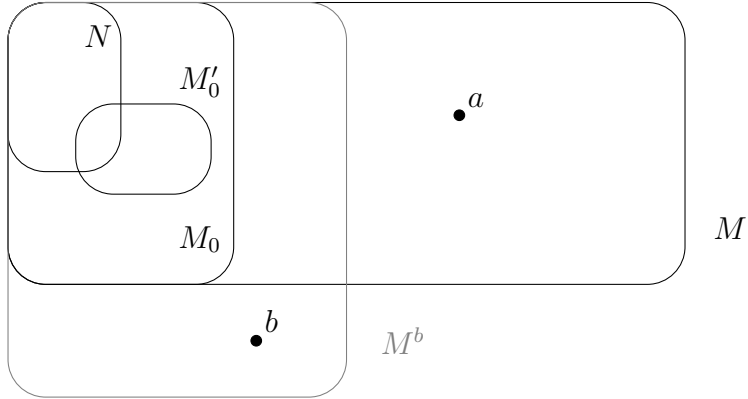
\begin{figure}[h]
\begin{tikzpicture}[rounded corners=5mm, scale=3,inner sep=.5mm]
\draw (0,1.25) rectangle (.5,.5);
\draw (.4,1.1) node {$N$};
\draw (0,0) rectangle (3,1.25);
\draw (0,1.25) rectangle (1,0);
\draw (.85,.2) node {$M_0$};
\draw (3.2, .25) node {$M$};
\draw (.85, .9) node {$M'_0$};
\draw (.3,.8) rectangle (.9,.4);
\draw[color=gray] (0,1.25) rectangle (1.5, -.5);
\node at (1.1,-.25)[circle, fill, draw, label=45:$b$] {};
\node at (2,.75)[circle, fill, draw, label=45:$a$] {};
\draw[color=gray] (1.75,-.25) node {$M^{b}$};
\end{tikzpicture}
\caption{A diagram of the models and elements in the definition of weak uniform $\mu$-symmetry.  We require that $\gtp (b / M)$ does not $\mu$-\emph{fork} over $M_0$ in the weak version, so there exists $M_0'$ such that $M_0$ is limit over $M_0'$ and $\gtp (b / M)$ does not $\mu$-split over $M_0'$} \label{fig:weak uniform sym}
\end{figure}

 We start by showing that assuming $\mu$-superstability this distinction is inessential, i.e.\ the two properties are equivalent. We will use the following characterization of symmetry:

\begin{fact}[Theorem 5 in \cite{vandieren-symmetry-v1}]\label{sym-charact}
  Assume that $\K$ is $\mu$-superstable. The following are equivalent:

  \begin{enumerate}
    \item\label{symmetry item} $\K$ has $\mu$-symmetry.
    \item\label{reduced are continuous} Reduced towers of size $\mu$ are continuous\footnote{See \cite{gvv-toappear-v1_2} for the definitions of these terms.}.
  \end{enumerate}
\end{fact}

\begin{lemma}\label{weak-unif-equiv}
  Assume that $\K$ is $\mu$-superstable. The weak uniform $\mu$-symmetry is equivalent to uniform $\mu$-symmetry which is equivalent to $\mu$-symmetry.
\end{lemma}
\begin{proof}
  That uniform $\mu$-symmetry is equivalent to $\mu$-symmetry is easy (it appears as \cite[Proposition 3.5]{vv-symmetry-transfer-v2}). Clearly, uniform implies weak uniform. Now assuming weak uniform symmetry, the proof of $(\ref{symmetry item})\Rightarrow(\ref{reduced are continuous})$ of Fact \ref{sym-charact} still goes through. The point is that whenever we consider $\gtp (b / M)$ in the proof, $M = \bigcup_{i < \delta} M_i$ for some increasing continuous $\seq{M_i : i < \delta}$ with $M_{i + 1}$ universal over $M_i$ for all $i < \delta$, and we simply use that by superstability $\gtp (b / M)$ does not $\mu$-split over $M_i$ for some $i < \delta$. However we also have that $\gtp (b / M)$ explicitly does not $\mu$-fork over $(M_i, M_{i + 1})$.

  Therefore reduced towers are continuous, and hence by Fact \ref{sym-charact} $\K$ has $\mu$-symmetry.
\end{proof}

Next we recall the definition of the order property in AECs \cite[Definition 4.3]{sh394}.

\begin{definition}
  Let $\alpha$ and $\lambda$ be cardinals. A model $M \in \K$ has the \emph{$\alpha$-order property of length $\lambda$} if there exists $\seq{\ba_i : i < \lambda}$ inside $M$ with $\ell (\ba_i) = \alpha$ for all $i < \lambda$, such that for any $i_0 < j_0 < \lambda$ and $i_1 < j_1 < \lambda$, $\gtp (\ba_{i_0} \ba_{j_0} / \emptyset) \neq \gtp (\ba_{j_1} \ba_{i_1} / \emptyset)$. 

  We say that \emph{$\K$ has the $\alpha$-order property of length $\lambda$} if some $M \in \K$ has it. We say that \emph{$\K$ has the $\alpha$-order property} if it has the $\alpha$-order property of length $\lambda$ for all cardinals $\lambda$.
\end{definition}

We will use two important facts: the first says that it is enough to look at length up to the Hanf number. The second that the order property implies instability.

\begin{definition}\label{hanf-def}
  For $\lambda$ an infinite cardinal, $\hanf{\lambda} := \beth_{(2^{\lambda})^+}$.
\end{definition}

\begin{fact}[Claim 4.5.3 in \cite{sh394}]\label{shelah-hanf}
  Let $\alpha$ be a cardinal. If $\K$ has the $\alpha$-order property of length $\lambda$ for all $\lambda < \hanf{\alpha + \LS (K)}$, then $\K$ has the $\alpha$-order property.
\end{fact}
\begin{fact}\label{stab-facts-op}
  If $\K$ has the $\alpha$-order property and $\mu \ge \LS (\K)$ is such that $\mu = \mu^{\alpha}$, then $\K$ is not stable in $\mu$.
\end{fact}
\begin{proof}
  By \cite[Claim 4.8.2]{sh394} (see \cite[Fact 5.13]{bgkv-v2} for a proof), there exists $M \in \K_\mu$ such that $|\gS^{\alpha} (M)| > \mu$. By \cite[Theorem 3.1]{longtypes-toappear-v2}, $\K$ is not stable in $\mu$.
\end{proof}

The following lemma appears in some more abstract form in \cite[Lemma 5.6]{bgkv-v2}.  The lemma says that if we assume that $p$ does not $\mu$-fork over $M$,
then in the definition of non-splitting (Definition \ref{def:splitting}) we can  replace the $N_\ell$ by arbitrary sequences in $N$ of length at most $\mu$. In the proof of Lemma \ref{sym-lem}, this will be used for sequences of length one.

\begin{lemma}\label{ns-lemma}
  Let $\mu \ge \LS (\K)$. Let $M \in \K_\mu$ and $N \in \K_{\ge \mu}$ be such that $M \lea N$. Assume that $\K$ is stable in $\mu$. If $p \in \gS (N)$ does not $\mu$-fork over $M$ (Definition \ref{mu-forking-def}), $a$ realizes $p$, and $\bb_1, \bb_2 \in \fct{\le \mu}{|N|}$ are such that $\gtp (\bb_1 / M) = \gtp (\bb_2 / M)$, then $\gtp (a \bb_1 / M) = \gtp (a \bb_2 / M)$.
\end{lemma}

\begin{proof}
Pick $N_0 \in \K_\mu$ containing $\bb_1 \bb_2$ with $M \lea N_0 \lea N$. Then $p \rest N_0$ does not $\mu$-fork over $M$.   Replacing $N$ by $N_0$ if necessary, we can assume without loss of generality that $N \in \K_\mu$. By definition of $\mu$-forking, there exists $M_0 \in \K_\mu$ such that $M_0 \lea M$ and $p$ does not $\mu$-split over $M_0$. By the extension and uniqueness property for $\mu$-splitting there exists $N'$ extending $N$ of cardinality $\mu$ so that $N'$ is universal over both $N$ and $M$, and $\gtp(a/N')$ does not $\mu$-split over $M_0$.  Since  $\gtp (\bb_1 / M) = \gtp (\bb_2 / M)$ and since $N'$ is universal over $N$, we can find $f:N\underset{M}\rightarrow N'$ so that $f(\bb_1)=\bb_2$.
Since $\gtp(a/N')$ does not $\mu$-split over $M_0$ we know $\gtp(f(a)/f(N))=\gtp(a/f(N))$.  By our choice of $f$ this implies that there exists $g\in\Aut_{f(N)}(\C)$ so that 
$g(f(a))=a$, $g\restriction M=\id_M$, and $g(\bb_2)=\bb_2$.  In other words $\gtp(f(a)\bb_2/M)=\gtp(a\bb_2/M)$.
Moreover $f^{-1}$ witnesses that $\gtp (a \bb_1 / M) = \gtp (f(a) \bb_2 / M)$, which we have seen is equal to $\gtp(a\bb_2/M)$.
\end{proof}

The next lemma shows that failure of symmetry implies the order property. The proof is similar to that of \cite[Theorem 5.14]{bgkv-v2}, the difference is that we use Lemma \ref{ns-lemma} and the equivalence between symmetry and weak uniform symmetry (Lemma \ref{weak-unif-equiv}).

\begin{lemma}\label{sym-lem}
  Let $\lambda > \mu \ge \LS (\K)$. Assume that $\K$ is superstable in every $\chi \in [\mu, \lambda)$. If $\K$ does \emph{not} have $\mu$-symmetry, then it has the $\mu$-order property of length $\lambda$.
\end{lemma}
\begin{proof}
  By Lemma \ref{weak-unif-equiv}, $\K$ does not have weak uniform $\mu$-symmetry. We first pick witnesses to that fact. Pick limit models $N, M_0, M \in \K_\mu$ such that $M$ is limit over $M_0$ and $M_0$ is limit over $N$. Pick $b$ such that $\gtp (b / M)$ does not $\mu$-fork over $M_0$, $a \in |M|$, and $\gtp (a / M_0)$ explicitly does not $\mu$-fork over $(N, M_0)$, and there does \emph{not} exist $M_b \in \K_{\mu}$ containing $b$ and limit over $M_0$ so that $\gtp (a / M_b)$ explicitly does not $\mu$-fork over $(N, M_0)$. We will show that $\sea$ has the $\mu$-order property of length $\lambda$.

  We build increasing continuous $\seq{N_\alpha : \alpha < \lambda}$ and $\seq{a_\alpha, b_\alpha, N_\alpha' : \alpha < \lambda}$ by induction so that for all $\alpha < \lambda$:

  \begin{enumerate}
  \item $N_\alpha, N_\alpha' \in \K_{\mu + |\alpha|}$.
  \item $N_0$ is limit over $M$ and $b \in |N_0|$.
  \item $\gtp (a_\alpha / M_0) = \gtp (a / M_0)$ and $a_\alpha \in |N_\alpha'|$.
  \item $\gtp (b_\alpha / M) = \gtp (b / M)$ and $b_\alpha \in |N_{\alpha + 1}|$.
  \item $N_\alpha'$ is limit over $N_\alpha$ and $N_{\alpha + 1}$ is limit over $N_{\alpha}'$. 
  \item $\gtp (a_\alpha / N_\alpha)$ explicitly does not $\mu$-fork over $(N, M_0)$ and $\gtp (b_\alpha / N_\alpha')$ does not $\mu$-fork over $M_0$.
\end{enumerate}

\paragraph{\underline{This is possible}}

Let $N_0$ be any model in $\K_\mu$ containing $M$ and $a$ and limit over $M$. At $\alpha$ limits, let $N_\alpha := \bigcup_{\beta < \alpha} N_\beta$. Now assume inductively that $N_\beta$ has been defined for $\beta \le \alpha$, and $a_\beta$, $b_\beta$, $N_\beta'$ have been defined for $\beta < \alpha$. By extension for splitting, find $q \in \gS (N_\alpha)$ that explicitly does not $\mu$-fork over $(N, M_0)$ and extends $\gtp (a / M_0)$. Let $a_\alpha$ realize $q$ and pick $N_\alpha'$ limit over $N_\alpha$ containing $a_\alpha$. Now by extension again, find $q' \in \gS (N_\alpha')$ that does not $\mu$-fork over $M_0$ and extends $\gtp (b / M)$. Let $b_\alpha$ realize $q'$ and pick $N_{\alpha + 1}$ limit over $N_\alpha'$ containing $b_\alpha$.

\paragraph{\underline{This is enough}}

We show that for $\alpha, \beta < \lambda$:

\begin{enumerate}
  \item\label{cond-0} $\gtp (a_\alpha b / M_0) \neq \gtp (a b / M_0)$
  \item\label{cond-1} If $\beta < \alpha$, $\gtp (ab / M_0) \neq \gtp (a_\alpha b_\beta / M_0)$.
  \item\label{cond-2} If $\beta \ge \alpha$, $\gtp (ab / M_0) = \gtp (a_\alpha b_\beta / M_0)$.
\end{enumerate}

For (\ref{cond-0}), observe that $b \in |N_0| \subseteq |N_\alpha|$ and $\gtp (a_\alpha / N_\alpha)$ explicitly does not $\mu$-fork over $(N, M_0)$. Therefore by monotonicity $N_\alpha$ witnesses that there exists $N_b \in \K_\mu$ containing $b$ and limit over $M_0$ so that $\gtp (a_\alpha / M_b)$ explicitly does not $\mu$-fork over $(N, M_0)$. By failure of symmetry and invariance, we must have that $\gtp (a_\alpha b / M_0) \neq \gtp (a b / M_0)$.

For (\ref{cond-1}), suppose $\beta < \alpha$.  We know that $\gtp (a_\alpha / N_\alpha)$ does not $\mu$-fork over $M_0$. Since $\beta < \alpha$, $b, b_\beta \in |N_\alpha|$ and $\gtp (b / M) = \gtp (b_\beta / M)$, we must have by Lemma \ref{ns-lemma} that $\gtp (a_\alpha b / M_0) = \gtp (a_\alpha b_\beta / M_0)$\footnote{This is really where we use the equivalence between uniform $\mu$-symmetry and weak uniform $\mu$-symmetry: if we only had failure of uniform $\mu$-symmetry, then we would only know that $\gtp (b / M)$ does not \emph{$\mu$-split} over $M_0$, so would be unable to use Lemma \ref{ns-lemma}.}. 

Together with (\ref{cond-0}), this implies $\gtp (a b / M_0) \neq \gtp (a_\alpha b_\beta / M_0)$.

To see (\ref{cond-2}), suppose $\beta \geq \alpha$ and recall that (by (\ref{cond-0})) $\gtp (a b / M_0) = \gtp (a_\beta b / M_0)$. We also have that $\gtp (b_\beta / N_\beta')$ does not $\mu$-fork over $M_0$. Moreover $\gtp (a / M_0) = \gtp (a_\alpha / M_0)$, and $a, a_\alpha \in N_{\beta}'$. By Lemma \ref{ns-lemma} again, $\gtp (a b_\beta / M_0) = \gtp (a_\alpha b_\beta / M_0)$. This gives us that $\gtp (ab / M_0) = \gtp (a_\alpha b_\beta / M_0)$.

Now let $\bar{d}$ be an enumeration of $M_0$ and for $\alpha < \lambda$, let $\bc_\alpha := a_\alpha b_\alpha \bar{d}$. Then (\ref{cond-1}) and (\ref{cond-2}) together tell us that the sequence $\seq{\bc_\alpha \mid \alpha < \lambda}$ witnesses the $\mu$-order property of length $\lambda$.
\end{proof}

\begin{theorem}\label{sym-from-superstab}
  Let $\mu \ge \LS (\K)$. Then there exists $\lambda < \hanf{\mu}$ such that if $\K$ is superstable in every $\chi \in [\mu, \lambda)$, then $\K$ has $\mu$-symmetry.
\end{theorem}
\begin{proof}
  If $\K$ is unstable in $2^{\mu}$, then we can set $\lambda := \left(2^{\mu}\right)^+$ and get a vacuously true statement; so assume that $\K$ is stable in $2^{\mu}$. By Fact \ref{stab-facts-op}, $\K$ does not have the $\mu$-order property. By Fact \ref{shelah-hanf}, there exists $\lambda < \hanf{\mu}$ such that $\K$ does not have the $\mu$-order property of length $\lambda$. By Lemma \ref{sym-lem}, it is as desired.
\end{proof}

\begin{remark}
  How can one obtain many instances of superstability as in the hypothesis of Theorem \ref{sym-from-superstab}? One way is categoricity, see the next section. Another way is to start with one instance of superstability and transfer it up using tameness. Indeed by \cite[Proposition 10.10]{indep-aec-v4}, if $\K$ is $\mu$-superstable and $\mu$-tame, then it is superstable in every $\mu' \ge \mu$. Thus Theorem \ref{sym-from-superstab} generalizes \cite[Theorem 6.4]{vv-symmetry-transfer-v2} which obtained $\mu$-symmetry from $\mu$-superstability and $\mu$-tameness.
\end{remark}

\section{Symmetry and categoricity}

In this section, we apply Theorem \ref{sym-from-superstab} to categorical classes. We will use notation from Chapter 14 of \cite{baldwinbook09} (recall from Definition \ref{hanf-def} that $\hanf{\lambda} := \beth_{(2^{\lambda})^+}$):

\begin{definition}
  For $n < \omega$, define inductively:

  \begin{enumerate}
    \item $H_0 := \LS (\K)$.
    \item $H_{n + 1} := \hanf{H_n}$.
  \end{enumerate}
\end{definition}

Throughout this section, we assume (in addition to amalgamation) that $\K$ has no maximal models. This is not a big deal because we can always take a tail of the AEC to obtain it:

\begin{fact}[Proposition 10.13 in \cite{indep-aec-v4}]\label{nmm-fact}
  If $\K$ is an AEC with amalgamation categorical in a $\lambda \ge H_1$, then there exists $\chi < H_1$ such that $\K_{\ge \chi}$ has no maximal models.
\end{fact}

\begin{hypothesis}
  $\K$ is an AEC with amalgamation and no maximal models.
\end{hypothesis}

The following powerful fact has its roots in \cite[Theorem 2.2.1]{shvi635}, where it was proven assuming the generalized continuum hypothesis instead of amalgamation. This is the main tool to obtain superstability from categoricity. Its proof relies on Ehrenfeucht-Mostowski models, which is why we assumed no maximal models.

\begin{fact}[The Shelah-Villaveces theorem, Theorem 6.3 in \cite{gv-superstability-v2}]\label{shvi}
  Let $\lambda > \mu \ge \LS (\K)$. If $\K$ is categorical in $\lambda$, then $\K$ is $\mu$-superstable.
\end{fact}

Combining this fact with Theorem \ref{sym-from-superstab}, we obtain symmetry from categoricity in a high-enough cardinal:

\begin{theorem}\label{sym-categ}
  Let $\mu \ge \LS (\K)$. Assume that $\K$ is categorical in a $\lambda \ge \hanf{\mu}$. Then $\K$ is $\mu$-superstable and has $\mu$-symmetry. 
\end{theorem}
\begin{proof}
  By the Shelah-Villaveces theorem (Fact \ref{shvi}), $\K$ is $\mu'$-superstable in every $\mu' \in [\mu, \lambda)$. By Theorem \ref{sym-from-superstab}, $\K$ has $\mu$-symmetry.
\end{proof}

Combining Theorem \ref{sym-categ} with consequences of symmetry, we deduce:

\begin{corollary}
  Let $\mu \ge \LS (\K)$. Assume that $\K$ is categorical in a $\lambda \ge \hanf{\mu}$. For any $M_0, M_1, M_2 \in \K_\mu$, if $M_1$ and $M_2$ are both limit over $M_0$, then $M_1 \cong_{M_0} M_2$.
\end{corollary}
\begin{proof}
  By Theorem \ref{sym-categ} and Fact \ref{sym-uq-lim}.
\end{proof}

\begin{corollary}\label{cor-mu-sat}
  Let $\mu > \LS (\K)$. Assume that $\K$ is categorical in a $\lambda \ge \sup_{\mu_0 < \mu} \hanf{\mu_0^+}$. Then the model of size $\lambda$ is $\mu$-saturated. 
\end{corollary}
\begin{proof}
  We check that the model of size $\lambda$ is $\mu_0^+$-saturated for every $\mu_0 \in [\LS (\K), \mu)$. Fix such a $\mu_0$. By Theorem \ref{sym-categ}, $\K$ is $\mu_0$-superstable, $\mu_0^+$-superstable, and has $\mu_0^+$-symmetry. By Fact \ref{union-sat-monica} $\Ksatp{\mu_0^+}$, the class of $\mu_0^+$-saturated models in $\K_{\ge \mu_0^+}$, is an AEC with Löwenheim-Skolem number $\mu_0^+$. Since it has arbitrarily large models, it must have a model of size $\lambda$, which is unique by categoricity.
\end{proof}

We immediately obtain a downward categoricity transfer. We will use:

\begin{fact}[The AEC omitting type theorem, II.1.10 in \cite{sh394}]\label{omitting-type}
  Let $\lambda > \chi \ge \LS (\K)$. If $\K$ is categorical in $\lambda$ and the model of size $\lambda$ is $\chi^+$-saturated, then every model in $\K_{\ge \hanf{\chi}}$ is $\chi^+$-saturated.
\end{fact}

\begin{corollary}\label{easy-downward}
  Let $\mu = \beth_\mu > \LS (\K)$. Assume that $\K$ is categorical in a $\lambda > \mu$. Then $\K$ is categorical in $\mu$.
\end{corollary}
\begin{proof}
  By Corollary \ref{cor-mu-sat}, the model of size $\lambda$ is $\mu$-saturated. By Fact \ref{omitting-type}, every model in $\K_{\ge \mu}$ is $\mu$-saturated. In particular, every model of size $\mu$ is saturated. By uniqueness of saturated models, $\K$ is categorical in $\mu$.
\end{proof}
\begin{remark}
  For Corollary \ref{easy-downward}, Fact \ref{nmm-fact} shows that the no maximal models hypothesis is not necessary.
\end{remark}
\begin{remark}
  Corollary \ref{easy-downward} appears as \cite[Theorem 10.16]{indep-aec-v4} with an additional assumption of $(<\mu)$-tameness.
\end{remark}

We can improve on Corollary \ref{easy-downward} using the more powerful downward transfer of \cite{sh394}. A key concept in the proof is the following variation on tameness (an important locality for Galois types isolated by Grossberg and the first author in \cite{tamenessone}). We use the notation in \cite[Definition 11.6]{baldwinbook09}.

\begin{definition}
  Let $\chi, \mu$ be cardinals with $\LS (\K) \le \chi \le \mu$. $\K$ is \emph{$(\chi, \mu)$-weakly tame} if for any saturated $M \in \K_\mu$, any $p, q \in \gS (M)$, if $p \neq q$, there exists $M_0 \in \K_{\chi}$ with $M_0 \lea M$ and $p \rest M_0 \neq q \rest M_0$.
\end{definition}

The importance of weak tameness is that it is known to follow from categoricity in a suitable cardinal: this appears as \cite[Main Claim II.2.3]{sh394} and a simplified improved argument  is in \cite[Theorem 11.15]{baldwinbook09}.

\begin{fact}\label{weak-tameness-from-categ-fact}
  Let $\lambda > \mu \ge H_1$. Assume that $\K$ is categorical in $\lambda$, and the model of cardinality $\lambda$ is $\mu^+$-saturated. Then there exists $\chi < H_1$ such that $\K$ is $(\chi, \mu)$-weakly tame.
\end{fact}

We obtain:

\begin{theorem}\label{weak-tameness-from-categ}
  Let $\mu \ge \LS (\K)$. Let $\lambda \ge \hanf{\mu^+}$. If $\K$ is categorical in $\lambda$, then there exists $\chi < H_1$ such that $\K$ is $(\chi, \mu)$-weakly tame.
\end{theorem}
\begin{proof}
  By Corollary \ref{cor-mu-sat}, the model of size $\lambda$ is $\mu^+$-saturated. Now apply Fact \ref{weak-tameness-from-categ-fact}.
\end{proof}

We can derive the promised strenghtening of Corollary \ref{easy-downward}. By the proof of \cite[Theorem 14.9]{baldwinbook09} (originally \cite[II.1.6]{sh394}):

\begin{fact}\label{downward-categ-fact}
  If $\K$ is categorical in a $\lambda > H_2$, $\K$ is $(\chi, H_2)$-weakly tame for some $\chi < H_1$, and the model of size $\lambda$ is $\chi$-saturated, then $\K$ is categorical in $H_2$.
\end{fact}

\begin{theorem}
  If $\K$ is categorical in a $\lambda \ge \hanf{H_2^+}$, then $\K$ is categorical in $H_2$.
\end{theorem}
\begin{proof}
  By Theorem \ref{weak-tameness-from-categ}, there exists $\chi < H_1$ such that $\K$ is $(\chi, H_2)$-weakly tame. By Corollary \ref{cor-mu-sat}, the model of size $\lambda$ is $\chi$-saturated. Now apply Fact \ref{downward-categ-fact}.
\end{proof}

\begin{remark}
  We can replace $H_2$ above by any collection cardinal, see \cite[Definition 14.5]{baldwinbook09} and the proof of Theorem 14.9 there.
\end{remark}

\section{Good frames and weak tameness}\label{good-frame-sect}

In \cite[Definition II.2.1]{shelahaecbook}\footnote{We do \emph{not} mention Shelah's axiom (B) requiring the existence of a superlimit model of size $\mu$. In fact several papers (e.g.\ \cite{jrsh875}) define good frames without this assumption. However the good frame that we build here is categorical in $\mu$, hence has a superlimit.}, Shelah introduces \emph{good frames}, a local notion of independence for AECs. This is the central concept of his book and has seen many other applications, such as a proof of Shelah's categoricity conjecture for universal classes \cite{ap-universal-v6}. A \emph{good $\mu$-frame} is a triple $\s = (\K_\mu, \nf, \Sbs)$ where:

\begin{enumerate}
  \item $\K$ is a nonempty AEC which has $\mu$-amalgamation, $\mu$-joint embedding, no maximal models, and is stable in $\mu$.
  \item For each $M \in \K_\mu$, $\Sbs (M)$ (called the set of \emph{basic types} over $M$) is a set of nonalgebraic Galois types over $M$ satisfying (among others) the \emph{density property}: if $M \lta N$ are in $\K_\mu$, there exists $a \in |N| \backslash |M|$ such that $\gtp (a / M; N) \in \Sbs (M)$.
  \item $\nf$ is an (abstract) independence relation on types of length one over models in $\K_\lambda$ satisfying the basic properties of first-order forking in a superstable theory: invariance, monotonicity, extension, uniqueness, transitivity, local character, and symmetry (we will not give their exact meaning here).
\end{enumerate}

As in \cite[Definition II.6.35]{shelahaecbook}, we say that a good $\mu$-frame $\s$ is \emph{type-full} if for each $M \in \K_\mu$, $\Sbs (M)$ consists of \emph{all} the nonalgebraic types over $M$. We focus on type-full good frames in this paper. Given a type-full good $\mu$-frame $\s = (\K_\mu, \nf, \Sbs)$ and $M_0 \lea M$ both in $\K_\mu$, we say that a nonalgebraic type $p \in \gS (M)$ \emph{does not $\s$-fork over $M_0$} if it does not fork over $M_0$ according to the abstract independence relation $\nf$ of $\s$. We say that a good $\mu$-frame $\s$ is \emph{on $\K_\mu$} if its underlying class is $\K_\mu$. 

It was pointed out in \cite{ss-tame-toappear-v3} (and further improvements in \cite[Section 10]{indep-aec-v4} or \cite[Theorem 6.12]{vv-symmetry-transfer-v2}) that tameness can be combined with superstability to build a good frame. At a meeting in the winter of 2015 in San Antonio, the first author asked whether weak tameness could be used instead. This is not a generalization for the sake of generalization because weak tameness (but not tameness) is known to follow from categoricity. As it turns out, the methods of \cite{vv-symmetry-transfer-v2} can be used to answer in the affirmative:

\begin{theorem}\label{good-frame-weak-tameness}
  Let $\lambda > \mu \ge \LS (\K)$. Assume that $\K$ is superstable in every $\chi \in [\mu, \lambda]$ and has $\lambda$-symmetry.
  
  If $\K$ is $(\mu, \lambda)$-weakly tame, then there exists a type-full good $\lambda$-frame with underlying class the saturated models in $\K_{\lambda}$.
\end{theorem}
\begin{proof}
  First observe that limit models in $\K_\lambda$ are unique (by Fact \ref{sym-uq-lim}), hence saturated. By Fact \ref{sym-transfer}, $\K$ has $\chi$-symmetry for every $\chi \in [\mu, \lambda]$. By Fact \ref{union-sat-monica}, for every $\chi \in [\mu, \lambda)$, $\Ksatp{\chi^+}$, the class of $\chi^+$-saturated models in $\K_{\ge \chi^+}$ is an AEC with $\LS (\Ksatp{\chi^+}) = \chi^+$. Therefore (see \cite[Lemma 6.7]{vv-symmetry-transfer-v2}) $\Ksatp{\lambda}$ is an AEC with $\LS (\Ksatp{\lambda}) = \lambda$. By the $\lambda$-superstability assumption, $\Ksatp{\lambda}_\lambda$ is nonempty, has amalgamation, no maximal models, and joint embedding. It is also stable in $\lambda$. We want to define a type-full good $\lambda$-frame $\s$ on $\Ksatp{\lambda}_\lambda$. We define forking in the sense of $\s$ ($\s$-forking) as follows: For $M \lea N$ saturated of size $\lambda$, a nonalgebraic $p \in \gS (N)$ does not $\s$-fork over $M$ if and only if it does not $\mu$-fork over $M$. 

Now most of the axioms of good frames are verified in Section 4 of \cite{ss-tame-toappear-v3}, the only properties that remain to be checked are extension, uniqueness, and symmetry. Extension is by Fact \ref{mu-forking-props}.(\ref{mu-forking-3}), and uniqueness is by uniqueness in $\mu$ (Fact \ref{mu-forking-props}.(\ref{mu-forking-1})) and the weak tameness assumption. As for symmetry, we know that $\lambda$-symmetry holds, hence we obtain the result by Section 3 of \cite{vv-symmetry-transfer-v2}.
\end{proof}
\begin{remark}
  If $\lambda = \mu^+$ above, then the hypotheses reduce to ``$\K$ is superstable in $\mu$ and $\mu^+$ and $\K$ has $\mu^+$-symmetry''.
\end{remark}

Of course we can now combine this construction with our previous results:

\begin{corollary}\label{good-frame-weak-tameness-2}
  Let $\lambda > \mu \ge \LS (\K)$. Assume that $\K$ is superstable in every $\chi \in [\mu, \hanf{\lambda})$. If $\K$ is $(\mu, \lambda)$-weakly tame, then there exists a type-full good $\lambda$-frame with underlying class the saturated models in $\K_{\lambda}$.
\end{corollary}
\begin{proof}
  Combine Theorem \ref{good-frame-weak-tameness} and Theorem \ref{sym-from-superstab}.
\end{proof}

We obtain that a good frame can be built from categoricity in a high-enough cardinal (of arbitrary cofinality).

\begin{corollary}\label{good-frame-categ}
  Let $\mu \ge H_1$. Assume that $\K$ has no maximal models and is categorical in a $\lambda > \mu$. If the model of size $\lambda$ is $\mu^+$-saturated (e.g.\ if $\cf (\lambda) > \mu$ or by Corollary \ref{cor-mu-sat} if $\lambda \ge \hanf{\mu^+}$), then there exists a type-full good $\mu$-frame with underlying class the saturated models in $\K_\mu$.
\end{corollary}
\begin{proof}
  By Fact \ref{weak-tameness-from-categ-fact}, there exists $\chi < H_1$ such that $\K$ is $(\chi, \mu)$-weakly tame. By the Shelah-Villaveces theorem (Fact \ref{shvi}), $\K$ is superstable in every $\chi' \in [\chi, \lambda)$. By the proof of Fact \ref{sym-categ-fact} (which only uses that the model of size $\lambda$ is $\mu^+$-saturated), $\K$ has $\mu$-symmetry. Now apply Theorem \ref{good-frame-weak-tameness} with $(\mu, \lambda)$ there standing for $(\chi, \mu)$-here.
\end{proof}

\bibliographystyle{amsalpha}
\bibliography{sym-op}

\end{document}